\documentclass[11pt]{article}
\usepackage[margin=1in]{geometry}
\usepackage{multicol}
\usepackage{amsmath}
\usepackage{amssymb}
\usepackage{amsthm}
\theoremstyle{plain}  
\newtheorem{theorem}{Theorem}[section]
\newtheorem{lemma}[theorem]{Lemma}
\newtheorem{proposition}[theorem]{Proposition}

\def\QEDmark{\ensuremath{\square}}

\theoremstyle{definition}
\theoremstyle{remark} 
\newtheorem{example}{Example}[section]
\newtheorem{remark}{Remark}[section]

\usepackage{cite}
\usepackage{hyperref}
\usepackage{graphicx}
\usepackage{pdfsync}
\usepackage{amssymb}  
\usepackage{subfigure}
\usepackage{scalefnt}

\newcommand{\nsphere}{{\mathcal S^{n-1}}}

\newcommand{\coloneqq}{\mathrel{\mathop:}=}

\usepackage{color}
\usepackage{xcolor}

\newcommand{\bachir}[1]{{\color{black}#1}}
\newcommand{\aaa}[1]{{\color{black}#1}}

 \author{
 Amir Ali Ahmadi and  Bachir El Khadir \thanks{The authors are with  the department of Operations Research and Financial Engineering at Princeton University. Email:
  \{\texttt{a\_a\_a}, \texttt{bkhadir}\}\texttt{@princeton.edu}.\newline
 This work was partially supported by the DARPA Young Faculty Award, the Young Investigator Award of the AFOSR, the CAREER Award of the NSF, the Google Faculty Award, and the Sloan Fellowship.}
 }

\date{\today}
\title{\LARGE \bf On Algebraic Proofs of Stability\\ for Homogeneous Vector Fields}
\begin{document}
\date{}
\maketitle

\begin{abstract}
We prove that if a homogeneous, continuously differentiable vector field is asymptotically stable, then it admits a  Lyapunov function which is the ratio of two polynomials (i.e., a \emph{rational function}). We further show that when the vector field is polynomial, the Lyapunov inequalities on both the rational function and its derivative have \emph{sum of squares certificates} and hence such a Lyapunov function can always be found by \emph{semidefinite programming}. This generalizes the classical \aaa{fact} that an asymptotically stable linear system admits a quadratic Lyapunov function which satisfies a certain linear matrix inequality. In addition to homogeneous vector fields, the result can be useful for showing local asymptotic stability of non-homogeneous systems by proving asymptotic stability of their lowest order homogeneous component.

\aaa{This paper also includes some negative results: We show that (i) in absence of homogeneity, globally asymptotically stable polynomial vector fields may fail to admit a global rational Lyapunov function, and (ii) in presence of homogeneity, the degree of the numerator of a rational Lyapunov function may need to be arbitrarily high (even for vector fields of fixed degree and dimension).} On the other hand, we also give a family of homogeneous polynomial vector fields that admit a low-degree rational Lyapunov function but necessitate polynomial Lyapunov functions of arbitrarily high degree. This shows the potential benefits of working with rational Lyapunov functions, particularly as the ones whose existence we guarantee have structured denominators and are not more expensive to search for than polynomial ones. 
\end{abstract}



\maketitle

\section{Introduction and outline of contributions}

We are concerned in this paper with a continuous time dynamical system

\begin{equation}
\label{eq:CT.dynamics}
\dot{x}=f(x),
\end{equation}
where \(f:\mathbb{R}^n\rightarrow\mathbb{R}^n\) is continuously differentiable and has an equilibrium at the origin, i.e., \(f(0)=0\). The problem of deciding asymptotic stability of equilibrium points of such systems is a fundamental problem in control theory. The goal of this paper is prove that if \(f\) is a homogeneous vector field (see the definition below), then asymptotic stability is equivalent to existence of a Lyapunov function that is the ratio of two polynomials (i.e., a rational function). We also address the computational question of finding such a Lyapunov function in the case where the vector field \(f\) is polynomial.

A scalar valued function \(p:\mathbb{R}^n\rightarrow\mathbb{R}\) is said to be \emph{homogeneous} of degree \(d > 0\) if it satisfies \(p(\lambda x)=\lambda^d p(x)\) for all \(x\in\mathbb{R}^n\) and all \(\lambda\in\mathbb{R}\). Similarly, we say that a vector field  \(f: \mathbb R^n \rightarrow \mathbb R^n\) is homogeneous  of degree \(d > 0\) if \(f(\lambda x) = \lambda^d f(x)\) for all \(x\in\mathbb{R}^n\) and all \(\lambda\in\mathbb{R}\). \aaa{Homogeneous vector fields have been extensively studied in the literature on nonlinear control; see e.g.  \cite{Stability_homog_poly_ODE}, \cite{Stabilize_Homog}, \cite[Sect. 57]{Hahn_stability_book}, \cite{homog.feedback}, \cite{Baillieul_Homog_geometry},\aaa{\cite{efimov_numerical_2017}, \cite{efimov_inclusions}}, \cite{Cubic_Homog_Planar}, \cite{collins1996algebraic}, \cite{argemi1968points}, \cite{cima1990algebraic}, \cite{HomogHomog}, \cite{homog.systems}, \cite{feyzmahdavian2014asymptotic}, \cite{Zubov1964methods}, \cite{Kawski1988stability}. These systems are not only of interest as is: they can also be used to study properties of related \emph{non-homogeneous} systems. For example, if one can show that the vector field corresponding to the lowest-degree nonzero homogeneous component of the Taylor expansion of a smooth nonlinear vector field is asymptotically stable, then the vector field itself will be locally asymptotically stable.}

We recall that the origin of (\ref{eq:CT.dynamics}) is said to be \emph{stable in the sense of Lyapunov} if for every \(\epsilon>0\), there exists a \(\delta=\delta(\epsilon)>0\) such that $$\|x(0)\|<\delta\ \Rightarrow \|x(t)\|<\epsilon, \ \ \forall t\geq0.$$ We say that the origin is \emph{locally asymptotically stable} if it is stable in the sense of Lyapunov and if there exists a scalar \(\hat \delta > 0\) such that $$\|x(0)\|< \hat \delta \ \Rightarrow \ \lim_{t\rightarrow\infty}x(t)=0.$$ The origin is \emph{globally asymptotically stable} if it is stable in the sense of Lyapunov and  \(\lim_{t\rightarrow\infty}x(t)=0\) for any initial condition in \(\mathbb{R}^n\). A basic fact about homogeneous vector fields is that for these systems the notions of local and global asymptotic stability are equivalent. Indeed, the values that a homogeneous vector field \(f\) takes on the unit sphere determines its value everywhere.

It is also well known that the origin of (\ref{eq:CT.dynamics}) is globally asymptotically stable if there exists a continuously differentiable Lyapunov function \(V: \mathbb R^n \rightarrow \mathbb R\) which is radially unbounded (i.e., satisfies \(V(x) \rightarrow \infty\) as $\|x\| \rightarrow \infty$), vanishes at the origin, and is such that

\begin{align}
\label{eqn:Lyapunov_ineq_pd}
V(x) > 0 &\quad \forall x \ne 0
\end{align}
\begin{align}
\label{eqn:Lyapunov_ineq_grad_pd}
-\langle \nabla V(x), f(x) \rangle > 0&\quad \forall x \ne 0.
\end{align}

\aaa{Throughout this paper, whenever we refer to a \emph{Lyapunov function}, we mean a function satisfying the aforementioned properties.} We say that $V$ is \emph{positive definite} if it satisfies (\ref{eqn:Lyapunov_ineq_pd}). When $V$ is a homogeneous function, the inequality (\ref{eqn:Lyapunov_ineq_pd}) can be replaced by $$V(x) > 0 \quad \forall x \in \nsphere,$$ where \(\nsphere\) here denotes the unit sphere of \(\mathbb R^{n}\). It is straightforward to check that a positive definite homogeneous function is automatically radially unbounded.

\aaa{The first contribution of this paper is to show that an asymptotically stable homogeneous and continuously differentiable vector field always admits a Lyapunov function which is a \emph{rational function} (Theorem \ref{thm:rat_lyap_for_hom_gas}). This is done by utilizing a well-known result on existence of homogeneous Lyapunov functions \cite{HomogHomog}, \cite{Hahn_stability_book}, \cite{Zubov1964methods}, \cite{Kawski1988stability} and proving a statement on simultaneous approximation of homogeneous functions and their derivatives by homogeneous rational functions (Lemma~\ref{lem:approx_hom_by_rational}).}

\subsection{Polynomial vectors fields}

\aaa{We pay special attention in this paper to the case where the vector field $f$ in (\ref{eq:CT.dynamics}) is polynomial.} Polynomial differential equations appear ubiquitously in applications---either as true models of physical systems or as approximations of other families of nonlinear dynamics---\aaa{and have received a lot of attention in recent years because of the advent of promising analysis techniques using sum of squares optimization  \cite{PhD:Parrilo}, \cite{PapP02}, \cite{PositivePolyInControlBook}, \cite{AndyPackard2003}, \cite{Chesi2010LMI}, \cite{ChesiHenrion2009SpecialIEEE}, \cite{kamyar2014polynomial}.} In a nutshell, these techniques allow for an automated search over (a subset of) polynomial Lyapunov functions of bounded degree using semidefinite programming. However, there are comparatively few converse results in the literature \aaa{(e.g. those in \cite{Peet.exp.stability}, \cite{Peet.Antonis.converse.sos.journal}, \cite{tac_converse_sos}, \cite{AAA2011Converse})} on guaranteed existence of such Lyapunov functions.

In \cite{AAA_MK_PP_CDC11_no_Poly_Lyap}, the authors prove that there are globally asymptotically stable polynomial vector fields (of degree as low as 2) which do not admit polynomial Lyapunov functions. \aaa{We show in this paper that the same example in \cite{AAA_MK_PP_CDC11_no_Poly_Lyap} does not even admit a rational Lyapunov function (Section \ref{sec:non-existence-rational-Lyapunov-function}). This counterexample justifies the homogeneity assumption of our Theorem \ref{thm:rat_lyap_for_hom_gas}.}

\aaa{Homogeneous polynomial vector fields of degree 1 are nothing but linear systems. In this case, it is well known that asymptotic stability is equivalent to existence of a (homogeneous) quadratic Lyapunov function (see e.g. \cite[Thm. 4.6]{Khalil:3rd.Ed}) and can be checked in polynomial time.}
Moving up in the degrees, one can show that homogeneous vector fields of even degree can never be asymptotically stable \cite[Sect. 17]{Hahn_stability_book}. When the degree of $f$ is odd and $\ge 3$, testing asymptotic stability of (\ref{eq:CT.dynamics}) is not a trivial problem. In fact, already when the degree of $f$ is equal to 3 (and even if we restrict $f$ to be a gradient vector field), the problem of testing asymptotic stability is known to be strongly NP-hard \cite{ahmadi2012difficulty}. This result rules out the possibility of a polynomial time or even pseudo-polynomial time algorithm for this task unless P=NP. One difficulty that arises here is that tests of stability based on linearization fail. Indeed, the linearization of $f$ around the origin gives the identically zero vector field. This means (see e.g. \cite[Thm. 4.15]{Khalil:3rd.Ed}) that homogeneous polynomial vector fields of degree $\ge 3$ are never exponentially stable. This fact is independently proven by Hahn in \cite[Sect. 17]{Hahn_stability_book}.

\aaa{Our main contribution in this paper is to show that a proof of asymptotic stability for a homogeneous polynomial vector field can always be found by \emph{semidefinite programming} (Theorem \ref{thm:rat_hom_sdp}). This statement follows from existence of a rational Lyapunov function whose numerator is a \emph{strictly sum of squares} homogeneous polynomial (see Section \ref{sec:sdp_for_rat_lyap} for a definition) and whose denominator is an even power of the 2-norm of the state. Our result generalizes the classical converse Lyapunov theorem for linear systems which corresponds to the case where the power of the strictly sum of squares polynomial in the numerator (resp. denominator) is two (resp. zero).}

\aaa{Our next contribution is a negative result: We show in Proposition \ref{prop:no.finite.bound} that unlike the case of linear systems, for homogeneous polynomial vector fields of higher degree, one cannot bound the degree of the numerator of a rational Lyapunov function as a function of only the degree (or even the degree and the dimension) of the input vector field.}
We leave open the possibility that the degree of \bachir{this numerator} can be bounded as a \emph{computable} function of \aaa{the coefficients of the input vector field}. Such a statement (if true), together with the fact that semidefinite feasibility problems can be solved in finite time \cite{ComplexitySDP}, would imply that the question of testing asymptotic stability for homogeneous polynomial vector fields is decidable. Decidability of asymptotic stability for  polynomial vector fields is an outstanding open question of Arnlod; see \cite{Arnold_Problems_for_Math}, \cite{dacosta_doria_arnold}, \cite{Arnold_algebraic_unsolve}.


In Section \ref{sec:advantages_ratioanal_lyap}, we show a curious advantage \aaa{that rational Lyapunov functions can sometimes have over polynomial ones.} \aaa{In Proposition \ref{prop:unboundedness.poly.lyap},} we give a family of homogeneous polynomial vector fields of degree 5 that all admit a low-degree rational Lyapunov function but require polynomial Lyapunov functions of arbitrarily high degree. We end the paper with some concluding remarks and future research directions in Section \ref{sec:conclusion}.

\section{Approximation of homogeneous functions by rational functions}
\label{sec:orgheadline4}

\newcommand\Hnorm{\bachir{\mathcal H}}
For a positive even integer \(k\), let \(\mathcal H_k(\mathbb R^n)\) denote the set of continuously differentiable homogeneous functions \(V: \mathbb R^n \rightarrow \mathbb R\) of degree \(k\). For a function  \(V \in \mathcal H_k(\mathbb R^n)\), we define the norm \(\|.\|_{\Hnorm}\) as
$$\|V\|_{\Hnorm} =  \max\left\{ \max_{x \in \nsphere} |V(x)|, \max_{x \in \nsphere} \|\nabla V(x)\|_2\right\}.$$

\bachir{We prove in this section that homogeneous rational functions are dense in $\mathcal H_k(\mathbb R^n)$ for the norm \(\|.\|_{\Hnorm}\)}. \aaa{We remark that there is an elegant construction by Peet~\cite{Peet.exp.stability} that approximates derivatives of any function that has continuous mixed derivatives of order $n$ by derivatives of a polynomial. In contrast to that result, the construction below requires the function to only be continuously differentiable and gives a \emph{homogeneous} approximating function of degree $k$. This property is important for our purposes.}

 



\aaa{
\begin{lemma}
\bachir{Let $k$ be a positive even integer. For any function \(V \in \mathcal H_k(\mathbb R^n)\)} and any scalar $\varepsilon>0$, there exist an even integer \(r\) and a homogeneous polynomial \(p\) of degree \(r+k\) such that $$\left|\left|V(x) - \frac{p(x)}{\|x\|^{r}}\right|\right|_{\Hnorm} \le \varepsilon.$$

\label{lem:approx_hom_by_rational}
\end{lemma}
}

\begin{proof}
Fix \(V \in \mathcal H_k(\mathbb R^n)\) and $\varepsilon > 0$. For every integer \(m\), define the Bernstein polynomial of order $m$ as
\scalefont{.95}
\begin{align*}
B_m(x) =\sum_{0 \le j_1,  \ldots,  j_n \le m} &V\left(\frac{2j_1}{m}-1, \ldots, \frac{2j_n}{m}-1\right)\\
\cdot &\prod_{s=1}^n {m \choose j_s}  \left(\frac{1+x_s}2\right)^{j_s}\left(\frac{1-x_s}2\right)^{m-{j_s}}.
\end{align*}
\normalsize

The polynomial \(B_m\) has degree \(nm\), and has the property that for \(m\) large enough, it satisfies 

\begin{equation}
\label{eqn:conv_bernstein}
\begin{aligned}
\sup_{\|x\| \le 1} |B_m(x) - V(x)| \le \frac{\varepsilon}{1+k},\\
\text{ and } \sup_{\|x\| \le 1} \|\nabla B_m(x) - \nabla V(x)\| \le \frac{\varepsilon}{1+k}.
\end{aligned}
\end{equation}
See \cite[Theorem 4]{Bernstein2016} for a proof. Let $m$ be fixed now and large enough for the above inequalities to hold. Since \(V(x)\) is an even function, the function $$C(x) := \frac{B_m(x)+B_m(-x)}{2}$$ also satisfies (\ref{eqn:conv_bernstein}).
Because \(C(x)\) is even, the function $$\tilde C(x) := \|x\|^kC\left(\frac{x}{\|x\|}\right)$$ is of the form \(\frac{p(x)}{\|x\|^r}\), where \(p(x)\) is a homogeneous polynomial and \(r\) is an even integer. Also, by homogeneity, the degree of \(p(x)\) is \(r+k\).

It is clear that \(C\) and \(\tilde C\) are equal on the sphere, so
$$\sup_{\|x\| = 1} |\tilde C(x) - V(x)| \le  \frac{\varepsilon}{1+k}.$$

We argue now that the gradient of \(\tilde C\) is close to the gradient of \(V\) on the sphere. For that, fix \(x \in \nsphere\). By Euler's identity for homogeneous functions
$$\langle \nabla\tilde C(x), x\rangle - \langle \nabla V(x), x\rangle = k (\tilde C(x) - V(x)).$$
Since $$|\tilde C (x) - V(x)| \le \varepsilon,$$ it is enough to control the part of the gradient that is orthogonal to \(x\). More precisely, let $$\pi_x(y) \coloneqq y - \langle x, y\rangle x$$ be the projection of a vector $y \in \mathbb R^n$ onto the hyperplane $T_x$ tangent to $\nsphere$ at the point $x$. The following shows that \(\nabla \tilde C\) and \(\nabla C\) are equal when projected on $T_x$:

\begin{align*}
\pi_x(\nabla \tilde C(x))
&= \pi_x\left(k \|x\|^{k-2} C\left(\frac{x}{\|x\|}\right) x\right)
\\&+\pi_x\left( \|x\|^k J\left(\frac{x}{\|x\|}\right)^T \nabla C\left(\frac{x}{\|x\|}\right)\right)
\\&= \pi_x(k  C(x) x + (I-xx^T) \nabla C(x))
\\&= \pi_x(\nabla C(x)).
\end{align*}
Here, the second equation comes from the fact that $\|x\|=1$ and that the Jacobian of $\frac{x}{\|x\|}$ is equal to $I-xx^T$ on $\nsphere$, and the third equation relies on the fact that the projection of vector proportional to $x$ onto $T_x$ is zero. Therefore,
\begin{align*}  \|\pi_x(\nabla \tilde C(x) - \nabla V(x))\|
&= \|\pi_x(\nabla  C(x) - \nabla V(x))\|
\\&\le\|\nabla C(x) - \nabla V(x)\|
\\&\le  \frac{\varepsilon}{1+k}.
\end{align*}
We conclude by noting that 

\begin{align*}
\|\nabla \tilde C(x) - V(x)\| &\le \|\pi_x(\nabla \tilde C(x) - \nabla V(x))\|
\\&+ |\langle x, \nabla \tilde C(x) - \nabla V(x) \rangle|
\\&\le \varepsilon.\end{align*}
\end{proof}

\section{Rational Lyapunov functions}
\label{sec:existence_rat_lyap}

\subsection{Nonexistence of rational Lyapunov functions}
\label{sec:non-existence-rational-Lyapunov-function}

It is natural to wonder whether globally asymptotically stable polynomial vector fields always admit a rational Lyapunov function. We show here that this is not the case, hence also justifying the need for the homogeneity assumption in the statement of our main result (Theorem~\ref{thm:rat_hom_sdp}).

It has been shown in \cite{AAA_MK_PP_CDC11_no_Poly_Lyap} that the polynomial vector field

\begin{equation}
\label{eq:dynamical-system-no-rational-Lyapunov}
\begin{array}{lll}
\dot{x}&=&-x+xy \\
\dot{y}&=&\ \ -y
\end{array}
\end{equation}
is globally asymptotically stable (as shown by the Lyapunov function $V(x, y) = \log(1+x^2) + y^2$) but does not admit a polynomial Lyapunov function. We prove here that this vector field does not admit a rational Lyapunov function either. \aaa{Intuitively, we show that solutions of (\ref{eq:dynamical-system-no-rational-Lyapunov}) cannot be contained within sublevel sets of rational functions because they can grow exponentially before converging to the origin.}





\bachir{More formally,} suppose for the sake of contradiction that the system had a Lyapunov function of the form
$$V(x, y) = \frac{p(x, y)}{q(x, y)},$$
 where \(p(x, y)\) and \(q(x, y)\) are polynomials. Note first that the solution to system (\ref{eq:dynamical-system-no-rational-Lyapunov}) from any initial condition \((x_0, y_0) \in \mathbb R^2\) can be written explicitly:

\begin{equation*}
\begin{array}{lll}
x(t)&=&x_0e^{-t}e^{y_0(1-e^{-t})} \\
y(t)&=& y_0e^{-t}.
\end{array}
\end{equation*}

In particular, a solution that starts from \((x_0, y_0) = (k, \alpha k)\) for \(\alpha, k > 1\) reaches the point \((e^{\alpha(k-1)}, \alpha)\) after time $$t^*=\log(k).$$
 As \(t^* > 0\), the function $V$ must satisfy $$V(x(t^*), y(t^*)) < V(x_0, y_0),$$
 $$\text{i.e., } \quad\quad \frac{p(e^{\alpha(k-1)}, \alpha)}{q(e^{\alpha(k-1)}, \alpha)} < \frac{p(k, \alpha k)}{q(k, \alpha k)}.$$

Fix \(\alpha > 1\) and note that since  \(V(x, \alpha) \rightarrow \infty\) as \(x \rightarrow \infty\), then necessarily the degree of \(x \rightarrow p(x, \alpha)\) is larger than the degree of \(x \rightarrow q(x, \alpha)\). We can see from this that the left-hand side of the above inequality grows exponentially in \(k\) while the right-hand side grows polynomially, which cannot happen.

\subsection{Rational Lyapunov functions for homogeneous dynamical systems}
\label{sec:orgheadline6}
\label{sec:construction-rational-Lyapunov}


We now show that existence of rational Lyapunov functions is necessary for stability of homogeneous vector fields.

\begin{theorem}
Let \(f\) be a homogeneous, continuously differentiable function of degree \(d\). Then the system \(\dot x = f(x)\) is asymptotically stable if and only if  it admits a Lyapunov function of the type
\begin{equation}
\label{eqn:rationa_lyap}
V(x) = \frac{p(x)}{(\sum_{i=1}^n x_i^2)^r},
\end{equation}
where  \(r\) is a nonnegative integer and \(p\) is a homogeneous (positive definite) polynomial of degree $2r+2$.
\label{thm:rat_lyap_for_hom_gas}

\end{theorem}

\begin{proof}
The ``if direction" of the theorem is a standard application of Lyapunov's theorem; see e.g. \cite[Thm. 4.2]{Khalil:3rd.Ed}.

For the ``only if" \bachir{direction}, suppose \(f\) is continuously differentiable homogeneous function of degree \(d\), and that the system \(\dot x = f(x)\) is asymptotically stable. A result of Rosier \cite[Thm. 2]{HomogHomog} (see also \cite[Thm. 57.4]{Hahn_stability_book} \cite[Thm. 36]{Zubov1964methods} \cite[Prop. p.1246]{Kawski1988stability}) implies that there exists a function \(W \in \mathcal H_2(\mathbb R^n)\) such that
\begin{align*}
W(x) > 0 &\quad \forall x \in \nsphere,\\
-\langle \nabla W(x), f(x) \rangle > 0 & \quad \forall x \in \nsphere.
\end{align*}

Since these inequalities are strict and involve continuous functions, we may assume that there exists a \(\delta > 0\) such that
$$W(x) \ge \delta \text{ and }  -\langle \nabla W(x), f(x) \rangle \ge \delta \; \forall x \in \nsphere.$$
Let $$f_\infty \coloneqq \max\{1, \max_{\|x\| = 1} \|f(x)\|\}.$$
Lemma \ref{lem:approx_hom_by_rational} proves the existence of a function \(V(x)\) of the form (\ref{eqn:rationa_lyap}) that satisfies
\begin{align*}
|V(x)-W(x)| &\le \frac{\delta}{2f_{\infty}} & \quad \forall x \in \nsphere,\\
\|\nabla V(x)-\nabla W(x)\| &\le \frac{\delta}{2f_\infty} & \quad \forall x \in \nsphere.
\end{align*}
Fix \(x \in \nsphere\). An application of the Cauchy-Schwarz inequality gives
\begin{align*}
&|\langle \nabla W(x), f(x) \rangle - \langle \nabla V(x), f(x) \rangle| \\&\le \|\nabla W(x)- \nabla V(x)\| \|f(x)\| \\&\le\frac{\delta}2.
\end{align*}
Therefore,
$$V(x) \ge \frac\delta2 \text{ and }  -\langle \nabla V(x), f(x) \rangle \ge \frac\delta2 \; \forall x \in \nsphere.$$
\end{proof}

\section{An SDP hierarchy for searching for rational Lyapunov functions}
\label{sec:sdp_for_rat_lyap}

For a rational function of the type in  \bachir{(\ref{eqn:rationa_lyap})} to be a Lyapunov function, we need the polynomial $V$ and 

\begin{align*}
- \dot V(x) &\coloneqq - \langle \nabla V(x), f(x) \rangle 
\\&= \frac{-\|x\|^2 \langle\nabla p(x),f(x)\rangle +2r p(x)\langle x,f(x)\rangle}{\|x\|^{2(r+1)}},
\end{align*}
to be positive definite. This condition is equivalent to the polynomials in the numerators of $V$ and $-\dot V$ being positive definite. 
\aaa{In this section, we prove a stronger result which shows that there always exists a rational Lyapunov function whose two positivity requirements have ``sum of squares certificates''. This is valuable because the search over this more restricted class of positive polynomials can be carried out via semidefinite programming while the search over all positive polynomials is NP-hard~\cite{PhD:Parrilo}.}


Recall that a homogeneous polynomial $h$ of degree $2d$ is a \emph{sum of squares} (sos) if it can be written as $h = \sum_{i} g_i^2$ for some (homogeneous) polynomials $g_i$. This is equivalent to existence of a symmetric positive semidefinite matrix $Q$ that satisfies

\begin{equation}
\label{eqn:sos_gram}
h(x) = m(x)^TQm(x) \quad \forall x,
\end{equation}
where $m(x)$ is the vector of all monomials of degree $d$. We say that $h$ is \emph{strictly sos} if it is in the interior of the cone of sos homogeneous polynomials of degree $2d$. This is equivalent to existence of a positive definite matrix $Q$ that satisfies (\ref{eqn:sos_gram}). Note that a strictly sos homogeneous polynomial is positive definite. We will need the following Positivstellensatz due to Scheiderer.

\aaa{
\begin{lemma}[Scheiderer \cite{Claus_Hilbert17}, \cite{PowersPositiveDefiniteForms}]
\label{thm:claus}
For any two positive definite homogeneous polynomials $h$ and $g$, there exists an integer $q_0 \ge 0$ such that the polynomial $hg^q$ is strictly sos for all integers $q \ge q_0$.
\end{lemma}
}

\begin{theorem}
  If a homogeneous polynomial dynamical system admits a rational Lyapunov function of the form
  $$V(x)= \frac{p(x)}{(\sum_i x_i^2)^{r}},$$
where \(p(x)\) is a homogeneous polynomial, then it also admits a rational Lyapunov function 
$$W(x) = \frac{\hat p(x)}{(\sum_i x_i^2)^{\hat r}},$$
where the numerators of \(W\) and  $-\dot W$ are both strictly sos homogeneous polynomials.
\label{thm:sos_lyap_for_rat_lyap}

\end{theorem}

\begin{proof}
The condition that \(V\) be positive definite is equivalent to \(p\) being positive definite. The gradient of \(V\) is equal to
\begin{align*}
\nabla V(x) &= \frac{\|x\|^{2r}\nabla p(x) - 2r\|x\|^{2r-2}p(x) x }{\|x\|^{4r}}
\\&=  \frac{\|x\|^2 \nabla p(x) - 2rp(x) x }{\|x\|^{2r+2}}.
\end{align*}

If we let $$s(x) \coloneqq  \|x\|^2 \nabla p(x) - 2rp(x) x,$$
then the condition that  \(-\langle \nabla V(x), f(x) \rangle\) be positive definite is equivalent to \(-\langle s(x), f(x) \rangle\) being positive definite.

We claim that there exists a positive integer $\hat q$, such that 
$$W(x) \coloneqq V^{\hat q}(x)$$
 satisfies the conditions of the theorem. Indeed, by applying \bachir{Lemma \ref{thm:claus}} with $g=h=p$, there exists $q_0$, such that $p^q$ is strictly sos for all integers $q \ge q_0$.

Let us now examine the gradient of a function of the type $V^q$. We have
\begin{align*}
\nabla V^q(x)
= q V^{q-1}(x) \nabla V(x)
= q \left(\frac{p(x)}{\|x\|^{2r}}\right)^{q-1} \frac{s(x)}{\|x\|^{2r+2}}.
\end{align*}
Hence,
  $$-\langle \nabla V^q(x), f(x) \rangle = \frac{q}{\|x\|^{2rq + 2}} p(x)^{q-1} \langle -s(x), f(x)\rangle.$$
Since the homogeneous polynomials $$p(x) \mbox{\ and\ } \langle -s(x), f(x)\rangle$$ are both positive definite, by \bachir{Lemma~\ref{thm:claus}}, there exists an integer $q_1$ such that $$p(x)^{q-1} \langle -s(x), f(x) \rangle$$ is strictly sos for all $q \ge q_1$. Taking $\hat q = \max\{q_0, q_1\}$ finishes the proof as we can let $$\hat p = p^{\hat q}, \hat r = r\hat q.$$
\end{proof}

If we denote the \bachir{degree} of $\hat p$ by $s$, then characterization (\ref{eqn:sos_gram}) of strictly sos homogeneous polynomials applied to the numerator of $W$ and its derivative tells us that there exist positive definite matrices $P$ and $Q$ such that
$$W(x) = \frac{\langle m(x), P m(x) \rangle}{\|x\|^{2\hat r}},$$
and
$$- \dot W(x) = \frac{\langle z(x), Q z(x) \rangle}{\|x\|^{2\hat r+2}},$$
where \(m(x)\) (resp. $z(x)$) here denotes the vector of monomials in \(x\) of degree \(\frac s2\) (resp. $\frac{s+d+1}2$).
Notice that by multiplying  $W$ by a positive scalar, we can assume without loss of generality that $P \succeq I$ and $Q \succeq I$.

Putting Theorem \ref{thm:rat_lyap_for_hom_gas} and Theorem \ref{thm:sos_lyap_for_rat_lyap} together, we get the main result of this paper.

\begin{theorem}
  \label{thm:rat_hom_sdp}
  A homogeneous polynomial dynamical system \(\dot x = f(x)\) of degree $d$ is asymptotically stable if and only if there exist a nonnegative integer \(r\), a positive even integer \(s\), with \(2r < s\), and symmetric  matrices \(P \succeq I\) and \(Q \succeq I\), such that

  \begin{equation}
    \label{eq:rat_hom_sdp}
    \begin{aligned}
      \langle z(x), Q z(x)\rangle &= -2 \|x\|^2  \langle J(m(x))^T P m(x), f(x)\rangle
      \\&+ 2r m(x)^TPm(x) \langle x, f(x) \rangle  \quad \forall x \in \mathbb R^n,
    \end{aligned}
  \end{equation}
  \noindent where \(m(x)\) (resp. $z(x)$) here denotes the vector of monomials in \(x\) of degree \(\frac s2\) (resp. $\frac{s+d+1}2$), and $J(m(x))$ denotes the Jacobian of $m(x)$.
\end{theorem}

For fixed integers $s$ and $r$ with $2r < s$, one can test for existence of matrices $P \succeq I$ and $Q \succeq I$ that satisfy (\ref{eq:rat_hom_sdp}) by solving a semidefinite program. This gives rise to a hierarchy of semidefinite programs where one tries increasing values of $s$, and for each $s$, values of $r \in \{0,\ldots,\frac{s}2-1\}$.

\section{A negative result on degree bounds}
\label{sec:non-existence-degree-bound}

The sizes of the matrices $P$ and $Q$ that appear in the semidefinite programming hierarchy we just proposed depend on $s$, but not $r$. This motivates us to study whether one can bound $s$ as a function of the dimension $n$ and the degree $d$ of the vector field at hand. In this section, we show that the answer to this question is negative. In fact, we prove a stronger result which shows that one cannot bound the degree of the numerator of a rational Lyapunov function based on $n$ and $d$ only (even if one ignores the requirement that the Lyapunov function and its derivative have sos certificates of positivity).

To prove this statement, we build on ideas by Bacciotti and Rosier~\cite{Bacciotti.Rosier.Liapunov.Book} to construct a family of 2-dimensional degree-3 homogeneous polynomial vector fields that are asymptotically stable but necessitate rational Lyapunov functions whose numerators have arbitrarily high degree.

\aaa{
\begin{proposition}
Let \(\lambda\) be a positive irrational real number and consider the following homogeneous cubic vector field parameterized by the scalar \(\theta\):

\scalefont{.82}
\begin{equation}
\label{eq:a.s.cubic.vec.arbitrary.high.Lyap}
\begin{pmatrix}\dot{x} \\ \dot{y}\end{pmatrix} =\begin{pmatrix}\cos(\theta) & -\sin(\theta)\\\sin(\theta)&\ \ \  \cos(\theta)  \end{pmatrix} \begin{pmatrix}-2\lambda y(x^2+y^2)-2y(2x^2+y^2) \\
\ \ 4\lambda x(x^2+y^2)+2x(2x^2+y^2)
\end{pmatrix}.
\end{equation}

\normalsize

Then, for $0 < \theta < \pi$, the origin is asymptotically stable. However, for any positive integer \(s\), there exits a scalar \(\theta \in (0, \pi)\)  such that the vector field in (\ref{eq:a.s.cubic.vec.arbitrary.high.Lyap}) does not admit a rational Lyapunov function with a homogeneous polynomial numerator of degree \(\leq s\) and a homogeneous polynomial denominator.
\label{prop:no.finite.bound}
\end{proposition}
}

\aaa{The intuition behind this construction is that as $\theta\rightarrow 0,$ this sequence of vector fields converges to a limit vector field whose trajectories are periodic orbits that cannot be contained within level sets of any rational function. This limit behavior is formalized in the next lemma, which will be used in the proof of the above proposition.}


\aaa{
\begin{lemma}\label{prop:Bacciotti.Rosier}
Consider the vector field
\begin{equation}
\label{eq:Bacciotti.Rosier.f0}
\begin{pmatrix}\dot x\\\dot y\end{pmatrix} = 
f(x, y) = \left\{\begin{array}{lll}
-2\lambda y(x^2+y^2)-2y(2x^2+y^2) \\
4\lambda x(x^2+y^2)+2x(2x^2+y^2)
\end{array}\right.
\end{equation}
parameterized by a scalar \(\lambda > 0\). For all values of \(\lambda\), the origin is a center\footnote{By this we mean that all trajectories of (\ref{eq:Bacciotti.Rosier.f0}) go on periodic orbits which form closed curves around the origin.} of (\ref{eq:Bacciotti.Rosier.f0}), but for any irrational value of \(\lambda\), there exist no two bivariate polynomials \(p\) and \(q\) such that the rational function $$W(x, y) \coloneqq \frac{p(x, y)}{q(x, y)}$$
is nonzero, homogeneous, differentiable, and satisfies 
$$\langle \nabla W(x, y), f(x, y) \rangle = 0 \quad \text{for all } (x, y) \in \mathbb R^2.$$
\end{lemma}}
\begin{proof}
For the proof of the first claim see \cite[Prop.5.2]{Bacciotti.Rosier.Liapunov.Book}. Our proof technique for the second claim is also similar to \cite[Prop.5.2]{Bacciotti.Rosier.Liapunov.Book}, except for some minor differences. Suppose for the sake of contradiction that such a function \(W(x, y)\) exists. Let \(k\) denote the degree of homogeneity of $W$.
We first observe that the function 
$$I(x, y) = (x^2+y^2) (2x^2+y^2)^{\lambda}$$
satisfies $\langle \nabla I(x, y), f(x, y) \rangle = 0.$
Therefore, on the level set $$\{ (x, y) \in \mathbb R^2 \; | \; I(x, y) = 1\},$$
\(W(x, y)\) must be equal to a nonzero constant \(c\).
A homogeneity argument shows that
$$W(x, y) = c I(x, y)^{\frac k{2(1+\lambda)}} \text{ for all } (x, y) \in \mathbb R^2.$$
Hence, by setting \(x = 1\),
\begin{equation}
\label{eq-restriction-p}
p(1, y) = c (1+y^2)^{\frac k{2(1+\lambda)}} (2 + y^2)^{\frac {k\lambda}{2(1+\lambda)}} q(1, y)\text{ for all } y \in \mathbb R.
\end{equation}
Let \(r\) be the largest nonnegative integer such that  $$q(1, y) = (1+y^2)^r \hat q(y),$$ where \(\hat q\) is a univariate polynomial. As a result, $\hat q$ must satisfy \(\hat q(i) \ne 0\), where $i = \sqrt{-1}$.
 Then, from (\ref{eq-restriction-p}), we conclude that

\begin{equation}
\label{eq-restriction-p2}
p(1, y) = c (1+y^2)^{r+\frac k{2(1+\lambda)}} (2 + y^2)^{\frac {k\lambda}{2(1+\lambda)}} \hat q(y)\text{ for all } y \in \mathbb R.
\end{equation}

The function \(y \rightarrow (2 + y^2)^{\frac {k\lambda}{2(1+\lambda)}}\hat q(y)\) can be prolonged to a holomorphic function on the open set $$O_1 \coloneqq \mathbb C \setminus \{y=i v |\  |v| \ge \sqrt{2} \} .$$  

Furthermore, since $\hat q(i) \ne 0$, there exists an open neighborhood $O_2$ of $i$ where $\hat q$ does not vanish. On the open set $O_1 \cap O_2$, the function $$y \rightarrow (2 + y^2)^{\frac {k\lambda}{2(1+\lambda)}}\hat q(y)$$ is holomorphic and does not vanish, and hence by (\ref{eq-restriction-p2}), the function $$y \rightarrow (1+y^2)^{r+\frac k{2(1+\lambda)}}$$ is also holomorphic on $O_1 \cap O_2$. As a consequence, there exist an integer \(\bar n\) and a number \(\alpha \in \mathbb C \setminus \{0\}\) such that $$\frac{(1+y^2)^{r+\frac k{2(1+\lambda)}}}{(y-i)^{\bar n}} \rightarrow  \alpha$$ as \(y \rightarrow i\). This implies that $$r+\frac{k}{2(1+\lambda)} = \bar n$$ and contradicts the assumption that \(\lambda\) is an irrational number.
\end{proof}

\noindent\hspace{2em}{\itshape \aaa{Proof of Proposition \ref{prop:no.finite.bound}:}}
Consider the positive definite Lyapunov function\footnote{This function is not a polynomial, which can be seen e.g. by noticing that the restriction \(V(x, x) = 3^{\lambda} 2 x^{2(\lambda+1)}\) is not a polynomial.}
$V(x,y)=(2x^2+y^2)^\lambda(x^2+y^2)$,
whose derivative along the trajectories of
(\ref{eq:a.s.cubic.vec.arbitrary.high.Lyap}) is equal to $$\dot{V}(x,y)=-\sin(\theta)(2x^2+y^2)^{\lambda-1}(\dot{x}^2+\dot{y}^2).$$
Since \(\dot{V}\) is negative definite for \(0<\theta<\pi\), it follows that for \(\theta\) in this range, the origin of (\ref{eq:a.s.cubic.vec.arbitrary.high.Lyap}) is asymptotically stable.

To establish the latter claim of the \aaa{proposition}, suppose for the sake of contradiction that there exists an upper bound \(\bar{s}\) such that for all \(0<\theta<\pi\) the system admits a rational Lyapunov function $$W_{\theta}(x, y) = \frac{p_{\theta}(x, y)}{q_{\theta}(x, y)},$$ where \(p_{\theta}\) and \(q_\theta\) are both homogeneous polynomials and \(p_{\theta}\) is of degree at most \(\bar s\) independently of \(\theta\). Note that as a Lyapunov function, \(W_{\theta}\) must vanish at the origin, and therefore the degree of \(q_{\theta}\) is also bounded by \(\bar s\).
By rescaling, we can assume without loss of generality that the 2-norm of the coefficients of all polynomials \(p_{\theta}\) and \(q_{\theta}\) is 1.

Let us now consider the sequences \(\{p_\theta\}\) and \(\{q_{\theta}\}\) as \(\theta\rightarrow 0\). These sequences reside in the compact set of bivariate homogeneous polynomials of degree at most \(\bar s\) with the 2-norm of the coefficients equal to 1. Since every bounded sequence has a converging subsequence, it follows that there must exist a subsequence of \(\{p_\theta\}\) (resp. \(\{q_\theta\}\)) that converges (in the coefficient sense) to some nonzero  homogeneous polynomial \(p_0\) (resp. \(q_0\)). Define $$W_0(x, y) \coloneqq \frac{p_0(x, y)}{q_0(x, y)}.$$
Since convergence of this subsequence also implies convergence of the associated gradient vectors, we get that
$$\dot{W}_0(x,y)= \langle \nabla W_{0}(x, y), \begin{pmatrix}\dot x\\\dot y\end{pmatrix} \rangle\leq0.$$
On the other hand, when \(\theta=0\), the vector field in (\ref{eq:a.s.cubic.vec.arbitrary.high.Lyap}) is the same as the one in (\ref{eq:Bacciotti.Rosier.f0}) and hence the trajectories starting from any nonzero initial condition go on periodic orbits. This however implies that \(\dot{W}_0=0\) everywhere and in view of \bachir{Lemma \ref{prop:Bacciotti.Rosier}} we have a contradiction.
\hspace{\fill}\QEDmark

\begin{remark}
It is possible to establish the result of Proposition~\ref{prop:no.finite.bound} without having to use irrational coefficients in the vector field. One approach is to take an irrational number, e.g. \(\pi\), and then think of a sequence of vector fields given by (\ref{eq:a.s.cubic.vec.arbitrary.high.Lyap}) that is parameterized by both \(\theta\) and \(\lambda\). We let the \(k\text{-th}\) vector field in the sequence have \(\theta_k=\frac{1}{k}\) and \(\lambda_k\) equal to a rational number representing \(\pi\) up to \(k\) decimal digits.
Since in the limit as \(k\rightarrow\infty\) we have \(\theta_k\rightarrow 0\) and \(\lambda_k\rightarrow \pi\), it should be clear from the proof of Proposition~\ref{prop:no.finite.bound} that for any integer \(s\), there exists an asymptotically stable bivariate homogeneous cubic vector field with \emph{rational} coefficients that does not have a Lyapunov \(\frac{p(x,y)}{q(x,y)}\) where \(p\) and \(q\) are homogeneous and \(p\) has degree less than \(s\).
\end{remark}


\section{\aaa{Potential} advantages of rational Lyapunov functions over polynomial ones}
\label{sec:advantages_ratioanal_lyap}

\aaa{In this section, we show that there are stable polynomial vector fields for which a polynomial Lyapunov function would need to have much higher degree than the sum of the degrees of the numerator and the denominator of a rational Lyapunov function. The reader can also observe that independently of the integer $r$, the size of the SDP arising form Theorem~\ref{thm:rat_hom_sdp} that searches for a rational Lyapunov function with a numerator of degree $s$ and a denominator of degree $2r$ is smaller than the size of an SDP that would search for a polynomial Lyapunov function $p$ of degree $s+2$ (by requiring $p$ and $-\dot{p}$ to be sums of squares), even when $p$ is taken to be homogeneous. Therefore, for some vector fields, a search for a rational Lyapunov function instead of a polynomial one can be advantageous.}

\aaa{
\begin{proposition}
\label{prop:unboundedness.poly.lyap}
Consider the following homogeneous polynomial vector field parameterized by the scalar \(\theta\):

\scalefont{.95}
\begin{equation}
\label{eq:a.s.deg5.vec.arbitrary.high.poly.Lyap}
\begin{pmatrix}
\dot x\\\dot y
\end{pmatrix} = f_{\theta}(x ,y) =2 R(\theta)\begin{pmatrix} x\, \left(\ \ x^4 + 2\, x^2\, y^2 - y^4\right)\\ y\, \left( - x^4 + 2\, x^2\, y^2 + y^4\right)\end{pmatrix},
\end{equation}
\normalsize
where \[
R(\theta) \coloneqq \begin{pmatrix}-\sin(\theta) & -\cos(\theta)\\\cos(\theta)& -\sin(\theta)\end{pmatrix}.
\]
Then, for $\theta \in (0, \pi)$, the vector field $f_{\theta}$ admits the following rational Lyapunov function
$$W(x, y) = \frac{x^4+y^4}{x^2+y^2}$$
and hence is asymptotically stable.
However, for any positive integer \(\bar s\), there exits a scalar \(\theta \in (0, \pi)\)  such that $f_{\theta}$ does not admit a polynomial Lyapunov function of degree \(\leq \bar s\).
\end{proposition}
}

\aaa{Once again, the intuition is that as  $\theta\rightarrow 0,$ $f_\theta$ converges to a vector field $f_0$ whose trajectories are periodic orbits. This time however, these orbits will exactly traverse the level sets of the rational function $W$ and cannot be contained within level sets of any polynomial.} 
%
Our proof will utilize the following independent lemma about univariate polynomials.

\begin{lemma}
  \label{lem:non_existence_poly}
  There exist no two univariate polynomials $\tilde p$ and $\tilde q$, with $\tilde q$ non-constant, that satisfy
  $$\tilde p(x^2) = \tilde q\left(\frac{x^4+1}{x^2+1}\right) \quad \forall x \in \mathbb R.$$
\end{lemma}
\begin{proof}
  Assume for the sake of contradiction that such polynomials exist. For every nonnegative scalar $u$, there exists a scalar $x$ such that $u = x^2$. Therefore,
  $$\tilde p(u) = \tilde q\left(\frac{u^2+1}{u+1}\right) \quad \forall u \ge 0.$$

  The expression above is an equality between two univariate rational functions valid on $[0, \infty)$. Since both rational functions are well-defined on $(-1, \infty]$, the equality holds on that interval as well:
  $$\tilde p(u) = \tilde q\left(\frac{u^2+1}{u+1}\right) \quad \forall u > -1.$$

  We get a contradiction by taking $u \rightarrow -1$ as the left hand side converges to $\tilde p(-1)$, while the right hand side diverges to $\infty$.
\end{proof}
\noindent\hspace{2em}{\itshape \aaa{Proof of Proposition \ref{prop:unboundedness.poly.lyap}:}}
  Let us first prove that $W$ is a rational Lyapunov function associated with the vector field $f_{\theta}$ whenever  \(\theta \in (0, \pi)\). It is clear that $W$ is positive definite and radially unbounded. A straightforward calculation shows that
    $$f_{\theta}(x, y) = R(\theta) \; (x^2+y^2)^2 \; \nabla W(x, y).$$
Hence,
  \begin{align*}
    -\dot W(x, y) &= - \langle \nabla W(x, y), f_{\theta}(x, y) \rangle
    \\&= - (x^2+y^2)^2 \nabla W(x, y)^T R(\theta) \nabla W(x, y)
    \\&=  \sin(\theta) (x^2+y^2)^2 \|\nabla W(x, y)\|^2.
    \end{align*}

    Note that the function $\|\nabla W\|^2$ is positive definite as 
    $$W(x, y) = \frac12 \langle \begin{pmatrix}x\\y\end{pmatrix}, \nabla W(x, y) \rangle$$
    and $W$ is positive definite. This proves that when $0 < \theta < \pi$, the vector field $f_{\theta}$ is asymptotically stable with $W$ as a Lyapunov function.

    To prove the latter claim of the \aaa{proposition}, suppose for the sake of contradiction that there exists an upper bound \(\bar s\) such that for all \(0<\theta<\pi\) the system admits a polynomial Lyapunov function of degree at most \(\bar s\). By an argument similar to that in the proof of \aaa{Proposition \ref{prop:no.finite.bound}}, there must exist some nonzero polynomial \(p_0\), with $p_0(0) = 0$, that satisfies
    $$\dot p_0(x, y) \coloneqq \langle \nabla p_0(x, y), f_0(x, y) \rangle \le 0 \quad \forall (x, y) \in \mathbb R^2.$$

    We claim that $p_0$ must be constant on the level sets of $W$. To prove that, consider an arbitrary positive scalar $\gamma$ and the level set $$M_\gamma \coloneqq \{(x, y) \in \mathbb R^2 \; | \; W(x, y) = \gamma\}.$$	
	Since $W$ is homogeneous and positive definite, $M_\gamma$ is closed and bounded. In addition, $f_0$ is continuously differentiable and does not vanish on $M_\gamma$. Moreover, trajectories starting in $M_\gamma$ remain in $M_\gamma$ as 
	$$\langle \nabla W(x, y), f_{0}(x, y) \rangle = \sin(0) (x^2+y^2)^2 \|\nabla W(x, y)\|^2 = 0.$$
	Hence, by the Poincar\'e-Bendixson Criterion \cite[Lem 2.1]{Khalil:3rd.Ed}, the set $M$ contains a periodic orbit of $f_0$. 
	
	Let $z_1, z_2 \in M_\gamma$. We know that the trajectory starting from $z_1$ must visit $z_2$. Since $\dot p_0 \le 0$,  we must have $p_0(z_1) \le p_0(z_2)$. Similarly, we must also have $p_0(z_2) \le p_0(z_1)$, and therefore $$p_0(z_1) = p_0(z_2).$$

    Since we now know that $p_0$ is constant on the level sets of $W$, there must exist a function $g: \mathbb R \rightarrow \mathbb R$ such that
    $$p_0(x, y) = g(W(x, y)) =g\left(\frac{x^4+y^4}{x^2+y^2}\right).$$
    This proves that $$p_0(x, y) = p_0(x, -y) = p_0(-x, y) = p_0(-x, -y).$$
    Therefore, there exists a polynomial $p$ such that
    \begin{equation}p_0(x, y) = p(x^2, y^2) = g\left(\frac{x^4+y^4}{x^2+y^2}\right).
    \label{eqn:p.contradiction}\end{equation}
    Setting $y=0$, we get that $p(x^2, 0) = g(x^2)$.
    Hence, $ p(u, 0)=g(u)$ for all $u \ge 0$. Taking $$u = \frac{x^4+y^4}{x^2+y^2},$$ the second equality in (\ref{eqn:p.contradiction}) gives
    $$p(x^2, y^2) = p\left(\frac{x^4+y^4}{x^2+y^2}, 0\right).$$
    Setting $y = 1$, we get that the polynomial $p$ satisfies
     $$p(x^2, 1) = p\left(\frac{x^4+1}{x^2+1}, 0\right).$$
    If we let 
    $\tilde p(x) \coloneqq p(x, 1) \text{ and } \tilde q(x) \coloneqq p(x, 0),$
    then in view of Lemma \ref{lem:non_existence_poly} and the fact that $\tilde q$ is not constant, we have a contradiction.
\hspace{\fill}\QEDmark

\begin{example}

Consider the vector field $f_{\theta}$ in (\ref{eq:a.s.deg5.vec.arbitrary.high.poly.Lyap}) with $\theta = 0.05$. One  typical trajectory of this vector field is depicted in Figure \ref{fig:rational_vector_field}. We use the modeling language YALMIP \cite{yalmip} and the SDP solver MOSEK \cite{mosek} to search for rational and polynomial Lyapunov functions for this vector field.

\begin{figure}[h]\label{fig:rational_vector_field}
  \vspace{-.5cm}
  \centering
  \includegraphics[width=0.4\textwidth]{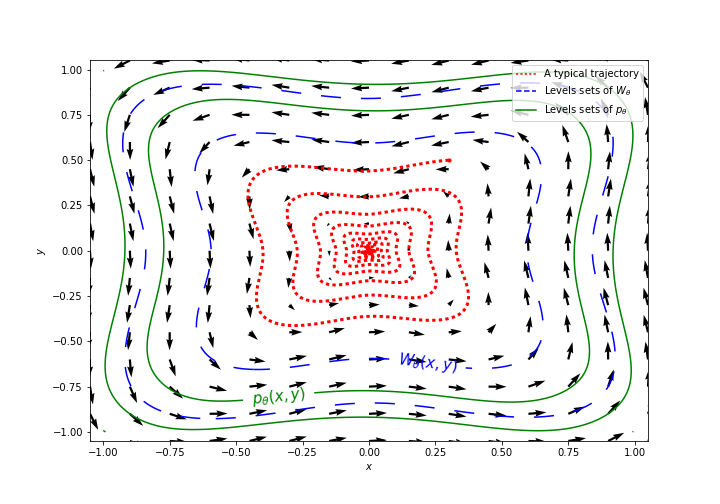}
  \vspace{-.5cm}
  \caption{A typical trajectory of the vector field $f_{\theta}$ in (\ref{eq:a.s.deg5.vec.arbitrary.high.poly.Lyap}) with $\theta = 0.05$, together with the level sets of the Lyapunov functions $W_\theta$ and $p_\theta$.}
\end{figure}

We know that for $\theta=0.05$, the vector field is asymptotically stable. Therefore, by Theorem \ref{thm:rat_hom_sdp}, the semidefinite programming \bachir{hierarchy} described in Section \ref{sec:sdp_for_rat_lyap} is guaranteed to find a rational Lyapunov function. The first round to succeed corresponds to $(s, r) = (4, 1)$, and produces the feasible solution
  \begin{align*}
    W_{\theta}(x, y) &= \frac{16.56x^4+16.56y^4+0.04x^2y^20.17x^3y-0.17xy^3}{x^2+y^2}
  \end{align*}
  If we look instead for a polynomial Lyapunov function, i.e. $r = 0$, the lowest degree for which the underlying SDP is feasible corresponds to $s=8$. The Lyapunov function that our solver returns is the following polynomial:
  \begin{align*}
    p_{\theta}(x, y) = 42.31x^8+42.31y^8+6.5xy^7-6.5x^7y-100.94x^2y^6
    \\-100.94x^6y^2+19.86x^5y^3-19.86x^3y^5+166.65 x^4y^4.
  \end{align*}
As all bivariate nonnegative homogeneous polynomials are sums of squares, infeasibility of our SDP for $s = 2, 4, 6$ means that $f_{\theta}$ admits no homogeneous polynomial Lyapunov function of degree lower than $8$. Two level sets of $W_{\theta}$ and $p_{\theta}$ are shown in Figure \ref{fig:rational_vector_field} and they look quite similar.
\end{example}


\section{Conclusions and future directions}
\label{sec:conclusion}
We showed in this paper that existence of a rational Lyapunov function is necessary and sufficient for asymptotic stability of homogeneous continuously differentiable vector fields. In the case where the vector field is polynomial, we constructed an SDP hiearachy that is guaranteed to find this Lyapunov function. The number of variables and constraints in this SDP hiearachy depend only on $s$, the degree of the numerator of the candidate Lyapunov function, and not on $r$, the degree of its denominator. To our knowledge, this theorem constitutes one of the few results in the theory of nonlinear dynamical systems which guarantees existence of algebraic certificates of stability that can be found by convex optimization (in fact, the only one we know of which applies to polynomial vector fields that are not exponentially stable). Regarding degree bounds, we proved that even for homogeneous polynomial vector fields of degree 3 on the plane, the degree $s$ of the numerator of such a rational Lyapunov function might need to be arbitrarily high. We also gave a family of homogeneous polynomial vector fields of degree 5 on the plane that all share a simple low-degree rational Lyapunov function, but require polynomial Lyapunov functions of arbitrarily high degree. Therefore, there are asymptotically stable polynomial vector fields for which a search for a rational Lyapunov function is much cheaper than a search for a polynomial one.
We leave the following two questions for future research:

\begin{itemize}
\item Can $r$ be upperbounded by a computable function of the coefficients of the vector field $f$? In particular, can $r$ always be taken to be zero? Or equivalently, do asymptotically stable homogeneous vector fields always admit a homogeneous polynomial Lyapunov function?
\item Similarily, can $s$ be upperbounded as a computable function of the coefficients of the vector field $f$? We have shown that $s$ cannot be upperbounded by a function of the dimension $n$ and the degree $d$ of the vector field only.
\end{itemize}

\aaa{Finally, while our focus in this paper was on analysis problems, we hope that our work also motivates further research on understanding the power and limitations of rational Lyapunov functions for \emph{controller design} problems.}


\section*{Acknowledgments} We are grateful to Tin Nguyen, the anonymous referees and the associate editor for several insightful comments.

\bibliographystyle{unsrt}
\bibliography{used_citations}

\end{document}